\documentclass[a4paper,11pt]{amsart}
\usepackage{amssymb,amsthm,bbm}

\DeclareMathOperator{\GL}{\mathrm{GL}}

\DeclareMathOperator{\Ker}{\mathrm{Ker}}
\DeclareMathOperator{\Hom}{\mathrm{Hom}}
\DeclareMathOperator{\rk}{\mathrm{rk}}
\DeclareMathOperator{\tr}{\mathrm{tr}}
\DeclareMathOperator{\dd}{\mathrm{d}}
\DeclareMathOperator{\End}{\mathrm{End}}
\DeclareMathOperator{\gl}{\mathfrak{gl}}
\def\Im{\mathrm{Im}\,}
\def\ii{\mathrm{i}}
\def\Id{\mathrm{Id}}
\def\C{\mathbbm{C}}
\def\R{\mathbbm{R}}
\def\Q{\mathbbm{Q}}
\def\Z{\mathbbm{Z}}
\def\la{\lambda}
\def\om{\omega}
\def\eps{\varepsilon}
\newtheorem{prop}{Proposition}[section]
\newtheorem{lemma}[prop]{Lemma}
\newtheorem{cor}[prop]{Corollary}
\newtheorem{theo}[prop]{Theorem}

\title{Reflection groups acting on their hyperplanes}
\author{Ivan Marin}
\date{August 31st, 2008}
\begin{document}

\maketitle

\bigskip
\begin{center}
Institut de Math\'ematiques de Jussieu \\
Universit\'e Paris 7 \\
175 rue du Chevaleret \\
F-75013 Paris
\end{center}
\bigskip

\noindent {\bf Abstract.}
After having established elementary results on the relationship between
a finite complex (pseudo-)reflection group $W \subset \GL(V)$ and its reflection
arrangement $\mathcal{A}$, we prove that
the action of $W$ on $\mathcal{A}$ is canonically related with
other natural representations of $W$, through a `periodic' family
of representations of its braid group. We also prove that,
when $W$ is irreducible, then the squares of defining
linear forms for $\mathcal{A}$ span the quadratic forms
on $V$, which imply $|\mathcal{A}| \geq n(n+1)/2$
for $n = \dim V$, and relate the $W$-equivariance
of the corresponding map with the period
of our family.
\bigskip


\bigskip



\section{Introduction}

Let $V$ be finite-dimensional $\C$-vector space,
$W \subset \GL(V)$ be a finite (pseudo-)reflection group
with corresponding hyperplane arrangement $\mathcal{A}$.
We assume that $\mathcal{A}$ is essential, meaning
that $\bigcap \mathcal{A} = \{ 0 \}$ and denote $n = \dim V$ the
rank of $W$. We recall that an arrangement $\mathcal{A}$ is called irreducible
if it cannot be written as $\mathcal{A}_1 \times \mathcal{A}_2$,
and that $W$ is called irreducible if it acts irreducibly on $V$.
A basic result can be written as follows

\medskip
\noindent (0) $\mathcal{A}$ is irreducible iff $W$ is irreducible.
\medskip

Steinberg showed that the exterior powers of $V$ are irreducible.
His proof is based on the encryption of irreducibility
in the connectedness of certain graphs. From
this approach, the following is easily deduced

\medskip
\noindent (1) If $W$ is irreducible, then it contains an \emph{irreducible} parabolic
subgroup.
\medskip

Although this result is probably well-known to experts and easily
checked, it does not seem to appear in print, and is a key tool for
the sequel.

We then consider the permutation $W$-module $\C \mathcal{A}$. A choice
of linear maps $\alpha_H \in V^*$ with kernel $H \in \mathcal{A}$
defines a linear map $\Phi : \C \mathcal{A} \to S^2 V^*$ through
$\alpha_H \mapsto \alpha_H^2$. This map can be chosen to be a morphism
of $W$-modules when $W$ is a Coxeter group. We prove

\medskip
\noindent (2) $\Phi$ is onto iff $W$ is irreducible
\medskip

\noindent meaning that each quadratic form on $V$ is a linear combination of
the quadratic forms $\alpha_H^2$, as soon as $W$ is irreducible. As
a corollary, we get

\medskip
\noindent (3) The cardinality of $\mathcal{A}$ is at least $n(n+1)/2$.
\medskip

\noindent This lower bound is better than the usual $|\mathcal{A}| \geq n/2$
of \cite{OT}, cor. 6.98, and is sharp, as $|\mathcal{A}| = n(n+1)/2$
when $W$ is a Coxeter group of type $A_n$.

We denote $d_H$ the order of the (cyclic) fixer in $W$ of $H \in \mathcal{A}$,
and define the distinguished reflection $s \in W$ to be the reflection in
$W$ with $H = \Ker(s-1)$ and additional eigenvalue $\zeta_H = \exp(2 \ii \pi/d_H)$.
We let $d : \mathcal{A} \to \Z$
denote $H \mapsto d_H$. We did not find the following in the
standard textbooks :

\medskip
\noindent (4) The data $(\mathcal{A},d)$ determines $W$.
\medskip

Letting $B$ denote the braid group associated to
$W$, we show that $\C \mathcal{A}$, considered as a linear representation
of $B$, can be deformed through a path in $\Hom(B,\GL(V))$ which
canonically connects $\C \mathcal{A}$ to other representations of $W$.
This turns out to provide a natural generalization of the action of Weyl groups
on their positive roots to arbitray reflection groups.

Finally, we prove that this path $h \mapsto R_h$ is periodic,
namely that $R_{h + \kappa(W)} \simeq R_h$ for some integer
$\kappa(W)$, with $\kappa(W) = 2$ when $W$ is a Coxeter group.
Moreover, $\kappa(W) = 2$ if and only if the morphism
$\Phi$ above can be chosen to be a morphism of
$W$-modules. In particular, we get

\medskip
\noindent (5) If $\kappa(W) = 2$ then the $W$-module $S^2 V^*$
is a quotient of $\C \mathcal{A}$.
\medskip

We emphasize the fact that the proofs presented here are elementary in the
sense that, except for one of the last results,
no use is made either of the Shephard-Todd classification
of pseudo-reflection groups, nor of the invariants theory of these
groups.

\section{Reflection groups and reflection arrangements}

We recall from \cite{OT} the following basic notions about reflection groups and hyperplane
arrangements.
An endomorphism $s \in \GL(V)$ is called a (pseudo-)reflection
if it has finite order and $\Ker(s-1)$ is an hyperplane of
$V$. A finite subgroup $W$ of some $\GL(V)$ which is generated by
reflections is called a (complex) (pseudo-)reflection group. The hyperplane arrangement associated
to it is the collection $\mathcal{A}$ of the reflecting hyperplanes
$\Ker(s-1)$ for $s$ a reflection of $W$. There is a natural
function $d : \mathcal{A} \to \Z, H \mapsto d_H$ which associates to each
$H \in \mathcal{A}$ the order of the subgroup of $W$ fixing
$H$. We let $\zeta_H = \exp(2 \ii \pi/d_H)$, and call
a reflection $s$ distinguished if its nontrivial eigenvalue is
$\zeta_H$, with $\Ker(s-1) = H$.

A nontrivial subgroup $W_0$ of $W$ is called \emph{parabolic} if it
is the fixer of some linear subspace of $V$. By a fundamental result
of Steinberg, this linear supspace lies inside some
intersection of reflecting hyperplanes, and $W_0$ is
also a reflection group in $\GL(V)$.

In general, a (central) hyperplane $\mathcal{A}$ arrangement is a finite
collection of linear hyperplanes in $V$. When $\mathcal{A}$
originates from a reflection group $W$, then $\mathcal{A}$ is called
a reflection arrangement. An arrangement $\mathcal{A}$ is called
essential if $\bigcap \mathcal{A} = \{ 0 \}$ ; for
two arrangements $\mathcal{A}_1, \mathcal{A}_2$ in $V_1,V_2$,
the arrangement $\mathcal{A}$ in $V = V_1 \times V_2$
is defined as $\{ H \oplus V_2 ; H \in \mathcal{A}_1 \} \cup
\{ V_1 \oplus H ; H \in \mathcal{A}_2 \}$ ; two arrangements in $V$ are
isomorphic if they are deduced one from the other by some element
of $\GL(V)$ ; an essential arrangement $\mathcal{A}$ is called irreducible if
it is not isomorphic to some nontrivial $\mathcal{A}_1 \times \mathcal{A}_2$.

The following lemma shows that, when $\mathcal{A}$ is a reflection
arrangement, the arrangement $\mathcal{A}$ together with
the order of the reflections determines the reflection group.
In particular, there is at most one reflection group with
reflections of order 2 admitting a given reflection arrangement. 
Notice that $\mathcal{A}$ can be assumed to be essential,
as the action of $W$ on $\bigcap \mathcal{A}$ is necessarily
trivial. Although basic, this fact does not appear in standard textbooks.
The proof given here has been found in common with Fran\c cois Digne
and Jean Michel.

\begin{prop} \label{propAdetermW} Let $\mathcal{A}$ be an essential hyperplane
arrangement in $V$.
\begin{enumerate}
\item If $P \in \GL(V)$ satisfies $P(H) \subset H$ for all $H \in \mathcal{A}$,
then $P$ is semisimple.
\item If $\mathcal{A}$ is a reflection arrangement associated
to a complex reflection group $W \subset \GL(V)$, then $(\mathcal{A},d)$ determines
$W$.
\end{enumerate}
\end{prop}
\begin{proof}
To prove (1), we choose linear forms $\alpha_H \in V^*$ with kernel $H \in \mathcal{A}$. Since
$\mathcal{A}$ is essential, $V^*$ is generated by the $\alpha_H$, hence admits a basis
made out some of them. The assumption then states that the $\alpha_H$ are eigenvectors
for $ ^t P \in \GL(V^*)$, hence $ ^t P$ is semisimple and so is $P$. Now we prove (2),
assuming that $W_1,W_2 \subset \GL(V)$ are two reflection groups with the same data
$(\mathcal{A},d)$. Let $H \in \mathcal{A}$ and $s_i \in W_i$
the distinguished reflection with $\Ker(s_i - 1) = H$. Then $x =s_1 s_2^{-1}$ fixes $H$
and acts by 1 on $V/H$, hence is unipotent. The endomorphism $x \in \GL(V)$
clearly permutes the hyperplanes. Since $\mathcal{A}$ is finite,
some power of $x$ setwise stabilizes every $H \in \mathcal{A}$,
hence is semisimple by (1). Since it is also unipotent
this power of $x$ is the identity, hence $x = \Id$ because
$x$ is unipotent. It follows that $s_1 = s_2$ hence $W_1 = W_2$.
\end{proof}

\section{A consequence of Steinberg lemma}

Let $W \subset \GL(V)$ be a reflection group and $\mathcal{A}$ the
corresponding reflection arrangement.
A basic fact is that the notions of irreducibility
for $W$ and $\mathcal{A}$ coincide and can be checked combinatorially on
some graph.
After recalling a proof of this, we notice a useful consequence.

We endow $V$ with a $W$-invariant hermitian scalar
product. Call $v \in V$ a \emph{root} if it is an eigenvector
of a reflection $s \in V$ such that $s.v \neq v$. For $L$
a finite set of linearly independent roots we let
$V_L$ denote the subspace of $V$ spanned by $L$, and
$\Gamma_L$ the graph on $L$ connecting $v_1$ and $v_2$
if and only if $v_1$ and $v_2$ are not orthogonal.
Notice that, if $s \in W$ is a reflection
with root $v \in V$, the following properties hold : if $v \in V_L$ then $s(V_L) \subset V_L$, because $V_L = (\C v) \oplus (\Ker(s-1) \cap V_L)$ ;
if $v \in V_L^{\perp}$ then $V_L \subset (\C v)^{\perp}$ is pointwise stabilized by $s$.

The following proposition is basic. We provide a proof of $(1) \Leftrightarrow (2)$ for
the convenience of the reader, because of a lack of reference. $(1) \Leftrightarrow (3)$ is due to Steinberg.

\begin{prop} The following are equivalent, for an essential reflection arrangement $\mathcal{A}$.
\begin{enumerate}
\item $W$ acts irreducibly on $V$.
\item $\mathcal{A}$ is an irreducible hyperplane arrangement.
\item $V$ admits a basis $L$ of roots such that $\Gamma_L$ is connected.
\end{enumerate}
\end{prop}
\begin{proof}
In the direction $(2) \Rightarrow (1)$, if $V = V_1 \oplus V_2$ with the $V_i$ being
$W$-stable subspaces, then we define $\mathcal{A}_i = \{ H \in \mathcal{A} \ | \ (s_H)_{|V_i} \neq \Id \}$
with $s_H$ the distinguished reflection w.r.t. $H \in \mathcal{A}$, and we have $\mathcal{A} = \mathcal{A}_1 \times \mathcal{A}_2$.
In the direction $(1) \Rightarrow (2)$, we let $V = V_1 \oplus V_2$ be the decomposition of
$V$ corresponding to $\mathcal{A} = \mathcal{A}_1 \times \mathcal{A}_2$.
We choose a collection of roots for $\mathcal{A}$. Let $s_1,s_2$
be two distinguished reflections associated to $H_1 \in \mathcal{A}_1,H_2 \in \mathcal{A}_2$,
respectively, and let $H = H_1 \oplus H_2 \subset V$. Consider some
reflection $s \in W$ such that $\Ker(s-1) \supset H$. If
$\Ker(s-1)$ can be written as $H_0 \oplus V_2$ with
$H_0$ some hyperplane of $V_1$, then $H_0 \oplus V_2 \supset
H_1 \oplus H_2$ implies $H_0 \supset H_1$, hence $H_0 = H_1$
by equality of dimensions, meaning that $s$ is some power of
$s_1$. Similarly, if $\Ker(s-1)$ can be written as $V_1 \oplus H_0$ with
$H_0$ some hyperplane of $V_2$, then $s$ is a power of $s_2$.
Considering the reflection $s_2 s_1 s_2^{-1}$, which fixes $H$
and has reflecting hyperplane $s_2.\Ker(s_1-1)$,
since $s_1 \neq s_2$ it follows that $s_2 s_1 s_2^{-1}$
is a power of $s_1$. Then
$s_2.\Ker(s_1 - 1) = \Ker(s_1-1)$
hence $s_1 ,s_2$ commute and have orthogonal roots.
The subspace $V_1^0$ spanned by all roots aring from $\mathcal{A}_1$
is thus setwise stabilized by all reflections of $W$, hence $V_1^0 = V$.
On the other hand, the hermitian scalar product induces an isomorphism
between $V_1^0$ and $V_1^*$ (because $\mathcal{A}_1$, like
$\mathcal{A}$, is essential), hence $V_2 \neq \{ 0 \} \Rightarrow V_1^0 \neq V$,
a contradiction.

We now prove $(1) \Leftrightarrow (3)$. Let $L_0$ be of maximal size among the sets $L$ of
linearly independent roots with connected $\Gamma_L$. We prove that $|L| = \dim V$ if $W$ is irreducible.
Indeed, since $W$ is irreducible generated by reflections and $V_{L_0} \subset V$,
there would otherwise exist a reflection $s$ such that $s(V_{L_0}) \not\subset V_{L_0}$.
Letting $v \in V$ be a root of $s$, we have $v \not\in V_{L_0}$ and
$v \not\in (V_{L_0})^{\perp}$. This proves that $L = L_0 \sqcup \{ v \}$
is made out linearly independant roots and that $\Gamma_L$ is connected, since
$v \not\in (V_{L_0})^{\perp}$ cannot be orthogonal to all roots spanning $L_0$
and $L_0$ is already connected. From this contradiction it follows that
$L_0$ has cardinality $\dim V$. Conversely, if $V$ admits a basis
$L$ of roots such that $\Gamma_L$ is connected, then $W$
is irreducible, for otherwise $V = V_1 \oplus V_2$ with
$V_1,V_2$ nontrivial orthogonal $W$-stable subspaces, and
$L = L_1 \sqcup L_2$ with $L_i = \{ x \in L \ | \ x \in U_i \}$.
Then $\Gamma_L = \Gamma_{L_1} \sqcup \Gamma_{L_2}$, contradicting
the connectedness of $\Gamma_L$.
\end{proof}

\begin{cor} \label{corparab} If $W \subset \GL(V)$ is an irreducible reflection group
then it admits an \emph{irreducible} parabolic subgroup
of rank $\dim V - 1$.
\end{cor}
\begin{proof}
Considering a set $L$ of linearly independent roots such that
$\Gamma_L$ is connected, as given by the proposition,
there exists $L_0 \subset L$ with $L = L_0 \sqcup \{ v \}$
such that $\Gamma_{L_0}$ is still connected. Then $V_{L_0}$
has dimension $\dim V - 1$, and its orthogonal is spanned
by some $v' \in V$. Letting $W_0$ denote the parabolic
subgroup fixing $v'$, it has rank $\dim V - 1$,
admits for roots all elements of $L_0$, hence is irreducible
since $\Gamma_{L_0}$ is connected.
\end{proof}

\section{Quadratic forms on $V$}

Let $\mathcal{A}$ be an essential hyperplane arrangement in $V$.
The integer $n = \dim V$ is the \emph{rank} $\rk \mathcal{A}$ of $\mathcal{A}$.
For each $H \in \mathcal{A}$ we let $\alpha_H \in V^*$ denote
some linear form with kernel $H$. For a field $\mathbbm{k}$, we let $\mathbbm{k} \mathcal{A}$ denote a vector
space with basis $v_H, H \in \mathcal{A}$, and define
a linear map $\Phi : \C \mathcal{A} \to S^2 V^*$ by $\Phi(v_H) = \alpha_H^2$.

For $\Phi$ to be onto, it is nessary that $\mathcal{A}$ is
irreducible. Indeed, if $\mathcal{A} = \mathcal{A}_1 \times \mathcal{A}_2$
corresponds to some direct sum decomposition $V = V_1 \oplus V_2$,
then choosing two nonzero linear forms $\varphi_i \in V_i^*$
defines a quadratic form $\varphi_1 \varphi_2 \in S^2 V^*$ which
does not belong to $\Im  \Phi$. This condition is also sufficient in rank 2.

\begin{prop} If $\mathcal{A}$ is essential of rank 2, then $\Phi$
is onto if and only if $\mathcal{A}$ is irreducible.
\end{prop}
\begin{proof}
Since $\mathcal{A}$ is essential, $\mathcal{A}$ contains at least two
hyperplanes $H_1,H_2$. We denote $\alpha_i = \alpha_{H_i}$ the corresponding
(linearly independant) linear forms. If $\mathcal{A} = \{ H_1, H_2 \}$,
then $\mathcal{A}$ is obviously reducible, so we may assume that
$\mathcal{A}$ contains at least another hyperplane. Let $\beta$ denote
the corresponding linear form. It can be written as $\beta = \la_1 \alpha_1
+ \la_2 \alpha_2$ with $\la_1 \neq 0$, $\la_2 \neq 0$. Since
$\beta^2 = \la_1^2 \alpha_1 ^2 + 2 \la_1 \la_2 \alpha_1 \alpha_2 + \la_2^2
\alpha_2^2$ and $\alpha_1^2,\alpha_2^2 , \beta^2 \in \Im  \Phi$
we get $\alpha_1 \alpha_2 \in \Im  \Phi$. Since $\alpha_1^2, \alpha_2^2 \in
\Im  \Phi$ and $\alpha_1,\alpha_2$ are linearly independent it
follows that $\Im  \Phi = S^2 V^*$.
\end{proof}

This condition is not sufficient in rank 3, as shows the following
example. Consider in $\C^3$ the central arrangement of polynomial
$xyz(x-y)(y-z)$. The morphism $\Phi$ is obviously not surjective,
as $\dim \C \mathcal{A} = 5$ and $\dim S^2 V^* = 6$. However, $\mathcal{A}$
is irreducible, because its Poincar\'e polynomial is $P_{\mathcal{A}}(t) = (1+t)(1+4t+4t^2)$,
which is not divisible by $(1+t)^2$  --- recall from \cite{OT} that $P_{\mathcal{A}_1 \times
\mathcal{A}_2} = P_{\mathcal{A}_1}  P_{\mathcal{A}_2 }$ and that
$P_{\mathcal{A}}(t)$ is divisible by $1+t$ whenever $\mathcal{A}$ is central.

It is however sufficient when $\mathcal{A}$ is a \emph{reflection arrangement}.

\begin{theo}
Let $\mathcal{A}$ be a (essential) reflection arrangement. Then
$\Phi$ is surjective if and only if $\mathcal{A}$ is irreducible.
\end{theo}
\begin{proof}
We assume that $\mathcal{A}$ is irreducible, and prove that $\Phi$
is surjective by induction on $\rk \mathcal{A}$.
If $\rk \mathcal{A} \leq 2$, this is a consequence of the above proposition,
so we can assume $\rk \mathcal{A} \geq 3$. We denote $W$
the corresponding (pseudo-)reflection group, and endow $V$ with a $W$-invariant
hermitian scalar product. By corollary \ref{corparab} there exists an irreducible
maximal parabolic subgroup $W_0 \subset W$, defined by
$W_0 = \{ w \in W \ | \ w.v = v \}$ for some $v \in V \setminus \{ 0 \}$.
We let $H_0 = (\C v)^{\perp}$. By Steinberg theorem $W_0$ is a reflection
group, whose (pseudo-)reflections are the reflections of $W$ contained
in $W_0$. Let $\mathcal{A}_0 \subset \mathcal{A}$ denote
the arrangement in $V$ corresponding to $W_0$. Since $v \in H$ 
for all $H \in \mathcal{A}_0$, by the induction hypothesis we have
$Q \subset S^2 H_0^*$, where $Q = \Im  \Phi$ and
$S^2 H_0^* \subset S^2 V^*$ is induced by $H^* \subset V^*$, letting
$\gamma \in H_0^*$ act on $H_0^{\perp}$ by 0.
Let $\alpha \in V^* \setminus \{ 0 \}$ such that $H_0 = \Ker \alpha$.
We have $S^2 V^* = S^2 H_0^* \oplus \alpha H_0^* \oplus \C \alpha^2$.
Since $\mathcal{A}$ is irreducible, there exists
$H \in \mathcal{A}$ such that $\alpha_H \not\in \C \alpha$
and $\alpha_H \not\in S^2 H_0^*$. Such a linear form can be written
$\la (\alpha + \beta)$ with $\la \in \C \setminus \{ 0 \}$ and
$\beta \in S^2 H_0^* \setminus \{ 0 \}$. Then $(\alpha + \beta)^2 \in Q$
and $\beta^2 \in Q$, so we have $\alpha^2 + 2 \alpha \beta \in Q$.
We make $W$ act on $V^*$ by $w.\gamma(x) = \gamma(w^{-1}.x)$,
for $x \in V$, $\gamma \in V^*$. Of course this action
can be restricted to a $W_0$-action on $H_0^* \subset V^*$. Then $w.(\alpha + \beta) \in Q$
for all $w \in W$, and since $w. \alpha = \alpha$ whenever
$w \in W_0$, we get $\alpha^2 + 2 \alpha(w.\beta) \in Q$
for all $w \in W_0$. Consider now the subspace $U$ of
$H^*$ spanned by the $w_1.\beta - w_2.\beta$ for $w_1,w_2 \in W_0$.
It is a $W_0$-stable subspace of $H_0^*$. Recall that $H_0$, hence
$H_0^*$, is irreducible under the action of $W_0$. If $U = \{ 0 \}$ then
$w.\beta = \beta$ for all $w \in W_0$, hence $H_0 = \C \beta$
and $\dim V = 2$, which has been excluded. Thus $U \neq \{ 0 \}$
hence $U = H_0^*$. By $2\alpha(w_1.\beta - w_2.\beta) =
(\alpha^2 + 2 \alpha (w_1.\beta)) - (\alpha^2 + 2 \alpha (w_2.\beta))$
we thus get $\alpha H_0^* \subset Q$. Then $(\alpha + \beta)^2 
\in \alpha^2 + \alpha H_0^* + S^2 H_0^* \subset \alpha^2 + Q$
implies $\alpha^2 \in Q$. It follows that $Q \supset S^2 V^*$
which concludes the proof.
\end{proof}

\begin{cor} If $\mathcal{A}$ is an irreducible reflection arrangement of
rank $n$, then $|\mathcal{A} | \geq n(n+1)/2$.
\end{cor}

Notice that the above lower bound is sharp, as it is reached
for Coxeter type $A_n$.

When $\mathcal{A}$ is a reflection arrangement
with corresponding reflection group $W$, both
$\C \mathcal{A}$ and $S^2 V^*$ can be endowed by natural $W$-actions,
where the action on $\C \mathcal{A}$ is defined by $w.v_H = v_{w(H)}$.
It is thus natural to ask whether the linear forms
$\alpha_H$ can be chosen such that $\Phi$ is a morphism
of $W$-modules.

\begin{prop} \label{propequivPhiCox}
If $\mathcal{A}$ is a complexified \emph{real} reflection arrangement
(in particular $W$ is a finite Coxeter group), then
the linear forms $\alpha_H$ can be chosen such that $\Phi$
is a morphism of $W$-modules.
\end{prop}
\begin{proof}
We choose a $W$-invariant scalar product on the original real form $V_0$ of
$V$ and extend it to a $W$-invariant hermitian scalar product on $V$.
For every $H \in \mathcal{A}$ we choose $x_H \in V_0$ orthogonal
to $H$ with norm 1, and define $\alpha_H : y \mapsto (x|y)$,
our convention on hermitian scalar products being that they are linear
on the right. Then, for any $w \in W$,  $w. x_H \in V_0$ is orthogonal
to $w(H)$ of norm 1, hence $w.x_H = \pm x_{w(H)}$. Since
$w. \alpha_H$ maps $y$ to $(w.x_H | y)$ we have $(w.\alpha_H)^2 = 
\alpha_{w(H)}^2$, which shows that $\Phi$ is a morphism
of $W$-modules.
\end{proof}

When $W$ is not a Coxeter group, the $W$-modules $\C \mathcal{A}$ and $S^2 V^*$ are
generally unrelated. However, this property is not a characterization
of Coxeter groups, as there is at least one example of a (non-Coxeter)
complex reflection group for which $\Phi$ can be a morphism
of $W$-module. This is the group labelled $G_{12}$ in the
Shephard-Todd classification. Notice that, in such a case,
one must have $\sum \alpha_H^2 = 0$, otherwise
this sum would provide a copy of the trivial
representation inside $S^2 V^*$, forcing $W$
to be a real reflection group.

We briefly describe this example. The group
$G_{12}$ can be described in $\GL_2(\C)$
by 3 generators $a,b,c$ of order 2, satisfying the
relation $abca=bcab=cabc$. We choose the following model :
$$
a = \begin{pmatrix} 1 & 1 + \sqrt{-2} \\ 0 & -1 \end{pmatrix}
b = \begin{pmatrix} -1 & 0 \\ 1-\sqrt{-2} & 1 \end{pmatrix}
c = \begin{pmatrix} \sqrt{-2} & -1 + \sqrt{-2} \\ -1 - \sqrt{-2} & -\sqrt{-2} \end{pmatrix}
$$
We define a collection of vectors $e_H \in V$, such that $w.e_H = \pm e_{w(H)}$. Letting
$\alpha_H : x \mapsto (e_H|x)$, the associated $\Phi : \C \mathcal{A} \to S^2 V^*$
is then a morphism of $W$-modules. A $W$-invariant hermitian scalar product is
given on this matrix model by $(X|Y) = ^t \bar{X} A Y$ with
$$
A = \begin{pmatrix} 2 & 1 + \sqrt{-2} \\ 1- \sqrt{-2}& 2 \end{pmatrix}
$$
We choose for $e_H$ the 12
following vectors, which are fixed by the
corresponding reflection $s$.

$$
\begin{array}{|c||c|c|c|}
\hline
s & babab & a&  b  \\ 
e_H & (1+\sqrt{-2},-2) &(1,0)&(0,1) \\
\hline
\hline
s &  ababa& bcb& c \\ 
e_H & (-2,1-\sqrt{-2})&(1,\sqrt{-2}) &(1,-1) \\
\hline
\hline
s & acaca & cbc & aba  \\
e_H & (1-\sqrt{-2},1+\sqrt{-2}) & (-1+\sqrt{-2},-\sqrt{-2}) & (-1-\sqrt{-2},1)  \\
\hline
\hline
s &  bab & cac & aca \\
e_H & (-1,1-\sqrt{-2})&
(-\sqrt{-2},1+\sqrt{-2})&(-\sqrt{-2},1) \\
\hline
\end{array}
$$
It can be checked that the
reflections $a,b,c$ act on these vectors by monomial matrices,
with nonzero entries in $\{ \pm 1 \}$ (hence factors through the
hyperoctahedral group of rank 12). On this example, $S^2 V^*$
is a selfdual $W$-module.

We make the following remark.

\begin{prop} \label{propPhiEqKap2} For $\Phi$ to be a morphism of $W$-modules
it is necessary that $\kappa(W) \leq 2$, where
$$
\kappa(W) = \min \{ n \in \Z_{>0} \ | \forall w \in W \ \forall H \in \mathcal{A} \ \ 
w.\alpha_H = \zeta \alpha_H \Rightarrow \zeta^n = 1 \} 
$$
\end{prop}

Using the Shephard-Todd classification, we will show in section 6
that this condition is actually sufficient when $W$ is irreducible.

\section{A path between representations}

In this section we define a natural connection between
the action of $W$ on $\C \mathcal{A}$ and more
surprising representations of $W$. For
this we need to introduce the space $X = V \setminus
\bigcup \mathcal{A}$ of regular vectors, on which
$W$ acts freely, and its quotient (orbit) space
$X/W$. We choose a base point $\underline{z} \in X$. The fundamental groups $B = \pi_1(X/W)$ and $P = \pi_1(X)$
are known as the braid group and pure braid group
associated to $W$, respectively. There is a natural morphism
$\pi : B \to W$ with kernel $P$. We first construct a deformation
of $W \to \GL(\C \mathcal{A})$ as a linear representation
of the braid group. This deformation should
not be confused with the one described in \cite{KRAMCRG} when
$W$ is a 2-reflection group.

\subsection{A representation of the braid group}

To each $H \in \mathcal{A}$ is canonically associated
a differential form $\om_H = \frac{\dd \alpha_H}{\alpha_H}$,
using some arbitrary linear form $\alpha_H$ with kernel
$\alpha_H$. We introduce idempotents
$p_H \in \End(\C \mathcal{A})$ defined by
$p_{H_1}.v_{H_2} = v_{H_2}$ if $H_1 = H_2$,
$p_{H_1}.v_{H_2} = 0$ otherwise. Choosing $h \in \C$,
the 1-form
$$
\om = h \sum_{H \in \mathcal{A}} p_H \om_H \in \Omega^1(X) \otimes \gl(\C \mathcal{A})
$$
satisfies $\om \wedge \om = 0$, hence defines a flat connection on the trivial
vector bundle $X \times \C \mathcal{A} \to X$, which is clearly $W$-equivariant
for the diagonal action on $X \times \C \mathcal{A}$. Dividing out by $W$,
the corresponding flat bundle over $X/W$
thus defines by monodromy a linear representation of $B$ in $\C \mathcal{A}$.
Letting $\gamma$ denote a representative loop of $\sigma \in B = \pi_1(X/W)$,
we can lift it to a path $\tilde{\gamma}$ in $X$ with endpoints
$\underline{z}$ and $\pi(\sigma).\underline{z}$, where $\underline{z}$
is the chosen basepoint in $X$. The 1-forms $\tilde{\gamma}^* \om_H$
can be written as $\gamma_H(t) \dd t$ for some function $\gamma_H$
on $[0,1]$, and the differential equation $\dd f = (\gamma^* \om)f$
to consider is then $f'(t) = h(\sum_{H \in \mathcal{A}} \gamma_H(t) p_H)f(t)$,
with $f(0) = \Id \in \End(\C \mathcal{A})$. Since the $p_H$ commute one to the other, the
solution is easy to compute :
$$
f(t) = \prod_{H \in \mathcal{A}} \exp\left( h p_H \int_0^t \gamma_H(u) \dd u \right)
$$
and the monodromy representation is given by
$$\sigma \mapsto R_h(\sigma) = \pi(\sigma) \prod_{H \in \mathcal{A}} \exp(h p_H \int_{\gamma} \om_H)$$
where we identified $w \in W$ with $R_0(w) \in \End(\C \mathcal{A})$. In particular, the image of $P$ is commutative. More precisely, if $\gamma_0$ is a loop in $X$
around a single hyperplane $H$, the class $[\gamma_0] \in P$ is mapped to $\exp(2 \ii \pi h p_H)$.
Since $P$ is generated by such classes, it follows that $R_n(P) = \{ \Id \}$ hence
$R_n$ factors through a representation of $W$ whenever $n \in \Z$.

We recall that $B$ is generated by so-called braided reflections
(`generators-of-the-monodromy' in \cite{BMR}), which are defined
as follows. For a distinguished reflection $s \in W$, an element $\sigma \in B$
with $\pi(\sigma) = s$ is called a braided reflection if it admits
as representative a path $\gamma$ from $\underline{z}$ to $s. \underline{z}$
which is a composite $(s.\gamma_0)^{-1}* \gamma_1 * \gamma_0  $
of paths with the following properties. Here $\gamma_0 : \underline{z} \rightsquigarrow 
\underline{z}_0$, $\gamma_1 : \underline{z_0} \rightsquigarrow 
s.\underline{z_0}$ and $(s.\gamma_0)^{-1} : s.\underline{z}_0 \rightsquigarrow 
s.\underline{z}$ is the reverse path of $s.\gamma_0$,
and $\gamma_1(t) = \eps \exp(2 \ii \pi t/d_H) \underline{z_0}^- + \underline{z_0}^+$
where $\underline{z_0}^+$ and $\underline{z_0}^-$ are the orthogonal
projection on $H$ and $H^{\perp}$, respectively, for $\eps>0$ small enough and
$\underline{z_0}$ sufficiently close to $H$ so that the homotopy class of this
path does not vary when $\eps$ decreases and $\underline{z_0}^+ \not\in H'$ for
$H' \in \mathcal{A} \setminus \{ H \}$.

Note that $\int_{s. \gamma_0} \om_{H'} = \int_{\gamma_0} \om_{s(H')}$
for all $H' \in \mathcal{A}$, hence $\int_{\gamma} \om_H =
\int_{\gamma_1} \om_H = (2 \ii \pi)/d_H$. In particular,
for such a braided reflection $\sigma$ we get

$$
R_h(\sigma).v_H = \pi(\sigma) \exp (h p_H \int_{\gamma} \om_H) v_H = \exp( 2 \ii \pi h /d_H) v_H.$$ 
\noindent Moreover, if $H$ and $H'$ have orthogonal
roots, then again $\int_{\gamma} \om_{H'} = \int_{\gamma_1} \om_{H'}$. But
in this case $\alpha_{H'}(\gamma_1(t))$ is constant hence $\int_{\gamma} \om_{H'} = 0$.
An immediate consequence of this is that we can restrict ourselves to irreducible groups, namely

\begin{prop} \label{propR1deco} If $W = W_1 \times \dots \times W_r$
is a decomposition of $W$ in irreducible components, with corresponding decompositions
$B = B_1 \times \dots \times B_k$ and $\mathcal{A} = \mathcal{A}^1 \times
\dots \mathcal{A}^r$,
then $R_h = R_{h}^{(1)} \times \dots \times R_{h}^{(r)}$ with $R_{h}^{(k)} : W_k \to \GL(\C \mathcal{A}^k)$.
\end{prop}

From the formulas above follows that, under the action of $R_h$,
$\C \mathcal{A}$ is the direct sum of the stable
subspaces $\C \mathcal{A}_k$, where $\mathcal{A} = \mathcal{A}_1
\sqcup \dots \sqcup \mathcal{A}_r$ is the decomposition
of $\mathcal{A}$ in orbits under the action of $W$. We let
$R_h^k : B \to \GL(\C \mathcal{A}_k)$, so that
$R_h = R_h^1 \oplus \dots \oplus R_h^r$.

\begin{prop} If $h \not\in \Z$, then $R_h^k$ is irreducible for each $1 \leq k \leq r$.
\end{prop}
\begin{proof}
For each $H \in \mathcal{A}_k$ we choose a loop $\gamma_H$
based at $\underline{z}$ around the hyperplane $H$,
We have $\int_{\gamma_H} \om_H = 2 \ii \pi$ and
$\int_{\gamma_H} \om_{H'} = 0$ for $H \neq H'$. Letting $Q_H$
denote the class of $\gamma_H$ in $P = \pi_1(X,\underline{z})$ 
we thus have $R_h^k(Q_H) = \exp(2 \ii \pi h p_H)$, hence
$R_h^k(Q_H) - \Id$ is a nonzero multiple of $p_H$ if $h \not\in \Z$.
It follows that the elements $R_h^k(Q_H)$ generate the commutative
algebra of diagonal matrices in $\End(\C \mathcal{A}_k)$.
Let $\mathcal{G}_k$ be the oriented graph
on the $v_H, H \in \mathcal{A}_k$ with an edge
$(v_{H_1},v_{H_2})$ if there exists $x \in B$ such that
the matrix $R_h^k(x)$ has nonzero entry at $(v_{H_1},v_{H_2})$.
If $\mathcal{G}_k$ is connected, then $R_h^k$ is irreducible
(see e.g. \cite{IRRED} prop. 3 cor. 2).
Choosing
for each distinguished reflection $s \in W$ a braided reflection $\sigma$,
$R_h^{k}(\sigma)$ has nonzero entries in $(v_H,v_{s(H)})$
and $(v_{s(H)},v_H)$ for each $H \in \mathcal{A}$. Since $\mathcal{A}_k$ is
an orbit under $W$ and $W$ is generated by distinguished reflections, it follows that
$\mathcal{G}_k$ is connected, concluding the proof.
\end{proof}

Since $R_h$ factors through $W$ when $h \in \Z$, this has the following consequence.

\begin{cor} \label{corsemisimple}
For all $h \in \C$, the representation $R_h$ of $B$ is semisimple.
\end{cor}

We choose a collection of roots $e_H, H \in \mathcal{A}$.
Notice that, for $w \in W$, $w(H) =H$ implies $w.e_H = e^{\ii \theta} e_H$
for some $\theta \in \R$.
\begin{lemma} \label{lemgammaom} If $\gamma : \underline{z} \rightsquigarrow w.\underline{z}$
is a path in $X$ with $w \in W$ such that $w.e_H = e^{\ii \theta} e_H$, then $\int_{\gamma} \om \in \ii \theta + 2 \ii \pi \Z$.
\end{lemma}
\begin{proof} We can assume $-\pi < \theta \leq \pi$. Since $\int_{\gamma} \om_H$
is independent of the choice of $\alpha_H$, we can choose $\alpha_H : x \mapsto (e_H|x)$
with $(e_H|e_H) = 1$. We have $\alpha_H(w.x) = e^{\ii \theta} \alpha(x)$. We write
$\gamma(t) = \gamma_H(t) + \gamma_0(t) e_H$ with $\gamma_0 : [0,1] \to \C$
and $\gamma_H : [0,1] \to H$. Then $\alpha_H(\gamma(t)) = \gamma_0(t)$
and $\int_{\gamma} \om_H = \int_{\gamma_0} \frac{\dd z}{z}$. Letting $x = \alpha_H(\underline{z}) \in \C^{\times}$,
we have $\gamma_0 : x \rightsquigarrow e^{\ii \theta} x$. If $\gamma_1 : x
\rightsquigarrow e^{\ii \theta}x$ is an arbitrary path in $\C^{\times}$, then $\gamma_0*\gamma_1^{-1}$
is a loop in $\C^{\times}$, hence $\int_{\gamma_0} \frac{\dd z}{z} - \int_{\gamma_1}
\frac{\dd z}{z}$ is a multiple of $2 \ii \pi$. If $e^{\ii \theta} = 1$
this concludes the proof. If $e^{\ii \theta} = -1$ we consider
$\gamma_1(t) = x e^{\ii \pi t}$, for which $\int_{\gamma_1} \frac{\dd z}{z} = \ii \pi$.
If $e^{\ii \theta} = \zeta \not\in \{1,-1 \}$ we consider $\gamma_1(t) = (1-t)x + te^{\ii \theta} x$
and $\int_{\gamma_1} \frac{\dd z}{z} = \left. \log(1 + (e^{\ii \theta} -1)t) \right|_0^1$
where $\log$ denotes the natural determination of the logarithm over $\C \setminus \R^-$.
It follows that $\int_{\gamma_1} \frac{\dd z}{z} = \log e^{\ii \theta} = \ii \theta$,
and the conclusion follows.
\end{proof}

We recall from section 4 the definition of $\kappa(W)$.
$$
\kappa = \kappa(W) = \min \{ n \in \Z_{>0} \ | \forall w \in W \ \forall H \in \mathcal{A} \ \ 
w.e_H = \zeta e_H \Rightarrow \zeta^n = 1 \}
$$

\begin{theo} \label{theoperiode}
For all $h \in \C$, $R_{h + \kappa}$ is isomorphic to $R_h$. Moreover,
$\kappa$ is the smallest positive real number such that $R_{\kappa} \simeq
R_0$.
\end{theo}
\begin{proof}
Recall from corollary \ref{corsemisimple} that,
for all $h \in \C$, $R_h$ is semisimple.
Letting $\chi_h$ denote the character of $R_h$ on $B$, it is thus sufficient
to prove $\chi_h = \chi_{h+\kappa}$ for all $h \in \C$ in order
to get $R_{h + \kappa} \simeq R_h$.
Let $g \in B$ with $w = \pi(g)$, and $\gamma : \underline{z} \rightsquigarrow w.\underline{z}$
a representing path. By the explicit formulas above, we
have
$$
\chi_h(g) = \sum_{w(H) = H} \exp(h \int_{\gamma} \om_H)
$$
and $R_{h+ \kappa} \simeq R_h$ follows by lemma \ref{lemgammaom}. We now show that $\kappa$ is minimal
with this property. Assuming otherwise, we let $0 < h < \kappa$
such that $\chi_h = \chi_0$. By definition of $\kappa$
there exists $w \in W$, $H \in \mathcal{A}$ such that $w.e_H = e^{\ii \theta} e_H$
with $e^{\ii \theta h} \neq 1$. Letting $g \in B$ with $\pi(g) = w$
and $\gamma : \underline{z} \rightsquigarrow w.\underline{z}$ a representing
path, we have $\int_{\gamma} \om_H \in \ii \theta + 2 \ii \pi \Z$,
hence $\exp(h \int_{\gamma} \om_H) \neq 1$. It follows that $|\chi_h(g)| < \chi_0(g)$
hence a contradiction.
\end{proof}

\begin{prop} For any $H \in \mathcal{A}$ and $h \in \C$,
if $\sigma$ is a braided reflection around $H$, then
$R_h(\sigma)$ is conjugated to $R_0(\sigma) \exp(h (2\ii \pi/d_H) p_H)$.
\end{prop}
\begin{proof}
Let $\sigma$ be a braided reflection with corresponding paths $\gamma, \gamma_0,\gamma_1$
as above. Since $\gamma_0$ and $s.\gamma_0$ represent the same path in $X/W$,
$R_h(\sigma)$ is conjugated to the monodromy along the loop $\gamma_1$ in $X/W$,
so that we can assume $\underline{z} = \underline{z_0}$, $\gamma = \gamma_1$.
In view of the formulas above, we thus only need to show that $\int_{\gamma_1} \om_{H'} = 0$
for $H' \neq H$. This can be done by direct computation,
as $\alpha_{H'}(\gamma_1(t)) = \eps \exp(2 \ii \pi t/d_H) \alpha_{H'}(\underline{z_0}^-) +
\alpha_{H'}(\underline{z_0}^+)$ with $\alpha_{H'}(\underline{z_0}^-) \neq 0$,
and $\int_{\gamma_1} \om_{H'}$ is constant when $\eps \to 0$. Since $\int_{\gamma_1} \om_{H'} \to 0$
when $\eps \to 0$ we get $\int_{\gamma} \om_{H'} = 0$ and the conclusion.
\end{proof}

\subsection{New representations of $W$}

When $n \in \Z$, the representation $R_n$ of $B$ factorizes through $W$.
In case $W$ is irreducible, the action of the center is easy to describe.

\begin{lemma} \label{lemR1centre} If $w \in W$ acts by $\la \in \C^{\times}$
on $V$, then $R_n(w) = \la^n \Id$ if $n \in \Z$. More generally,
if there exists $v \in X$ such that $w.v = \la v$ for some $\la \in \C^{\times}$,
then $R_n(w)$ is conjugated to $\la^n R_0(w)$
\end{lemma}
\begin{proof}
We first assume that $w$ acts on $V$ by $\la$. We can write $\la = \exp(\ii \theta)$ with $0 < \theta \leq 2 \pi$. We consider
the loop $\gamma(t) = e^{\ii \theta t} \underline{z}$ in $X/W$, whose image
in $W$ is $w$. By direct
calculation we have $\int_{\gamma} \om_H = \ii  \theta$ for all $H \in \mathcal{A}$
and the conclusion follows from the general formula for $R_1$. 
Now assume $w.v = \la v$ for some $\la = \exp(\ii \theta)$ with $0 < \theta \leq 2 \ii \pi$.
Up to conjugation, we can assume $v = \underline{z}$, the loop
$\gamma(t) = e^{\ii \theta t} \underline{z}$ in $X/W$ has image $w$ in $W$ and we conclude as before.
\end{proof}

More involved tools prove the following.

\begin{prop} \label{proprestrparab} If $W_0$ is a parabolic subgroup of $W$ with hyperplane
arrangement $\mathcal{A}$ and $n \in \Z$, then
the restriction of $R_n$ to $W_0$ is isomorphic to the direct sum
of the representation $R_n$ of $W_0$ and the permutation representation
of $W_0$ on $\C(\mathcal{A} \setminus \mathcal{A}_0)$.
\end{prop}
\begin{proof}
We let $R_h^0$ denote the representation $R_h$ for $W_0$ acting on
$\C \mathcal{A}_0$, and $S_h$ the direct sum of $R_h^0$ and
the permutation representation of $W_0$ on $\mathcal{A}\setminus
\mathcal{A}_0$. We can embed the braid group $B_0$ of $W_0$
inside $B$ such that, as representations over
$\C[[h]]$, the restriction to $B_0$ of $R_h$ is isomorphic to
$S_h$ (see \cite{KRAMCRG}, theorem 2.9). In particular, for all $g \in B_0$,
the traces of $R_h(g)$ and $S_h(g)$ are equal, as formal series in $h$.
Since these traces are holomorphic functions in $h$, it follows that
they are equal for all $h \in \C$. This means that the semisimple
representations of $B_0$ associated to the restriction of $R_h$ and to $S_h$
are isomorphic. Since the restriction of $R_n$ and $S_n$
are semisimple for all $n \in \Z$
the conclusion follows.

\end{proof}

The determination of the action of the center
enables us to prove that, contrary to $R_0$, $R_1$ is faithful in general.

\begin{prop} { \ \ }
\begin{enumerate}
\item $R_0$ has kernel $Z(W)$.
\item $R_1$ is faithful on $W$.
\item $\Ker R_n = \{ w \in Z(W) \ | \ w^n = 1 \}$
\end{enumerate}
\end{prop}
\begin{proof}
Without loss of generality (because of proposition \ref{propR1deco}) we may assume that $W$ is irreducible.
Obviously $(3) \Rightarrow (2)$.
Although (1) is also a special case of (3), we prove it separately. If $|\mathcal{A}| = 1$ the statement is obvious, so
we assume $|\mathcal{A}| \geq 2$. Clearly $Z(W) \subset \Ker R_0$,
as $\Ker(wgw^{-1} - 1) = w.\Ker(g-1)$ for all $g,w \in W$. Let
$w \in W$ such that $R_0(w) = \Id$, that is $w(H) = H$ for all $H \in
\mathcal{A}$. Let $s \in W$ be a distinguished reflection with
reflection hyperplane $H$. Then $wsw^{-1}$ is a reflection with
$\Ker(wsw^{-1} -1) = H$ which has the same nontrivial eigenvalue as $s$,
hence $wsw^{-1}  = s$.
It follows that $w$ commutes to all distinguished reflections of $W$,
hence $w \in Z(W)$ since $W$ is generated by such elements.

We now prove (3). Let $w \in \Ker R_n$. Since $R_1(w) = R_0(w) D$
for some diagonal matrix $D$, the nonzero entries of $R_n(w)$ determine
the permutation matrix $R_0(w)$, hence $w \in Z(W)$. Since $W$ is irreducible,
$w$ acts on $V$ by some scalar $\la \in \C^{\times}$, hence $R_n(w) = \la^n = 1$
by lemma \ref{lemR1centre}, hence $w^n = 1$. The converse inclusion is obvious
by lemma \ref{lemR1centre}.
\end{proof}

\begin{cor} The exponent of $Z(W)$ divides $\kappa(W)$.
If $W$ is irreducible then $|Z(W)|$ divides $\kappa(W)$.
\end{cor}
\begin{proof}
By the proposition, the period of the sequence $\Ker R_n$
is the exponent of $Z(W)$. Since $\Ker R_n$
is $\kappa(W)$-periodic the conlusion follows. If
$W$ is irreducible then $Z(W)$ is cyclic hence
its order equals its exponent.
\end{proof}

In the proof of theorem \ref{theoperiode}, we computed the character $\chi_n$
of $R_n$. We recall the result here :

\begin{prop} For any $w \in W$ and $n \in \Z$ we have
$$
\chi_n(w) = \sum_{w.e_H = \zeta e_H} \zeta^n
$$
\end{prop}

If $\tilde{K} = \Q(\zeta_d)$ is a cyclotomic field containing all eigenvalues of
$R_1(W)$, then letting $c_n \in \mathrm{Gal}(\tilde{K}|\Q)$ for $n \wedge d = 1$ be defined by
$c_n(\zeta_d) = \zeta_d^n$ we get from this proposition that $\chi_n = c_n \circ \chi_1$
for all $n$ prime to $d$.

\medskip

As an illustration of this section,
we do the example of $W$ of type $G_4$ generated
by
$$
s = \begin{pmatrix} 1 & 0 \\ 0 & j \end{pmatrix} \ \ 
t = \frac{1}{3}\begin{pmatrix} 1+2j & j-1 \\ 2j-2 & j+2 \end{pmatrix}.
$$
It is a reflection group of order 24, with two generators
$s,t$ of order 3 satisfying $sts=tst$, and center of order 2. It admits 3 one-dimensional
(irreducible) representations $S_{\alpha} : s,t \mapsto \alpha$, 3 two-dimensional
representations $A_{\alpha}$ with $\tr A_{\alpha}(s) = -\alpha$
for $\alpha \in \{ 1 , j,j^2 \}$ with $j = \exp(2 \ii \pi/3)$
and a 3-dimensional one that we denote $U$.
The reflection representation is $A_{j^2}$,
and $\kappa(W) = 6$. From the character table of $W$ one gets
$$\begin{array}{lcllcllcl}
R_0 &= & S_1 + U & R_1 &= & A_1 + A_{j^2} & R_2 & = & S_{j^2} + U \\
R_3 &=& A_j + A_{j^2} & R_4 &=& S_j + U & R_5 & = & A_1 + A_{j^2}
\end{array}
$$

\subsection{The case of Coxeter groups}

If $W$ is a Coxeter group, we get a simpler form of this representation.
Recall that, in this case, $\mathcal{A}$ is the
complexification of some real arrangement $\mathcal{A}_0$ in $V_0$,
where $V_0$ is a real form of $V$ ; moreover, choosing
some connected component $\mathcal{C}$ of $V_0 \setminus \bigcup \mathcal{A}_0$,
called a Weyl chamber, determines $n$ hyperplanes $H_1,\dots,H_n$ called
the \emph{walls} of $\mathcal{C}$, and the corresponding $n$ reflections $s_1,\dots,s_n$
are called the simple reflections associated to $\mathcal{C}$.
If $\underline{z} \in \mathcal{C}$, there is also a special
set of generators for $B$, namely the braided reflections $\sigma_i$ around
$H_i$ such that $\gamma_0$ is a straight (real) segment orthogonal to $H_i$.
These are called the Artin generators of $B$ (associated to a choice
of Weyl chamber).

\begin{prop} If $W$ is a Coxeter group 
with simple reflections $s_1,\dots,s_n$,
then $\sigma_i \mapsto R_0(s_i) \exp( \ii \pi h p_{H_i})$
defines a representation of $B$ which is equivalent to $R_h$. In particular,
$R_1$ is equivalent to a representation of $W$ on $\C \mathcal{A}$
for which $s_i.v_H = v_{s(H)}$ is $H \neq H_i$, $s_i.v_{H_i} = -v_{H_i}$,
and $R_{h+2}$ is equivalent to $R_h$ for any $h \in \C$, while $R_1 \not\simeq R_0$.
\end{prop}
\begin{proof}
We introduce the Weyl chamber $\mathcal{C} \subset V_0$ with respect to the
simple reflections $s_1,\dots,s_n$, with
walls 
$H_i = \Ker(s_i - 1)$, $1 \leq i \leq n$.
Up to conjugacy the base point $\underline{z}$ can be chosen inside the Weyl chamber,
and we define roots $e_H \in V_0$ of norm 1 such that
$\C  e_H = \Ker(s-1)^{\perp}$ and $(e_H | \underline{z}) > 0$
for $\underline{z} \in \mathcal{C}$. We choose for $\alpha_H$
the linear form $x \mapsto (e_H|x)$. Let us denote $\log^+$
the complex logarithm on $\C \setminus \ii \R_-^{\times}$,
and define
$$
D_h = \prod_{H \in \mathcal{A}} \exp(\ii \pi p_H\log^+ (e_H|\underline{z}))
$$
We consider a simple reflection $s_i$ around a wall $H_i$. Then
the path $\gamma$ representating $\sigma_i$ can be chosen with $\eps$
small enough so that $(e_{H}|\gamma(t))$ has positive real part
for each $t \in [0,1]$ and $H \neq H_i$. It follows that $t \mapsto \log^+(e_H|\gamma(t))$
has differential $\gamma^*\om_H$ and $R_h(\sigma_i)$ equals
$$
R_0(s_i) \prod_{H \in \mathcal{A}} \exp(h p_H \int_{\gamma} \om_H)
= R_0(s_i) \prod_{H \in \mathcal{A}} \exp\left( h p_H (\log^+ (e_H|s_i . \underline{z}) - \log^+ (e_H|\underline{z}))\right)
$$
(see \cite{KRAMCRG}, lemma 7.10). Moreover, $(e_H|s_i.\underline{z}) = (s_i.e_H | \underline{z})
= (e_{s_i(H)}|\underline{z})$ if $H \neq H_i$ (see e.g. \cite{KRAMCRG}, lemma 7.9)
and $(e_{H_i}|s_i.\underline{z}) = -(e_{H_i}|\underline{z})$. It follows that
$$
R_h(\sigma_i) = s_i \exp( \ii \pi h p_{H_i})\prod_{H \in \mathcal{A}\setminus\{ H_i \}} 
\exp\left( h p_H (\log^+ (e_{s_0(H)}|\underline{z}) - \log^+ (e_H|\underline{z}))\right)
$$
namely
$$
R_h(\sigma_i) = D_h s_i \exp( \ii \pi h p_{H_i}) D_h^{-1}
$$
for all $i \in [1,n]$, which concludes the proof. $R_1 \not\simeq R_0$
because $\tr R_1(s_1) = \tr R_0(s_1) - 1$.
\end{proof}

The representation of $W$ described in this proposition for $h = 1$ is natural in the realm of root systems.
Indeed, if a set $\mathcal{P}$ of roots for $\mathcal{A}_0$ is chosen,
such that $\mathcal{P}$ satisfies the axioms $(SR)_{I}$ and $(SR)_{II}$ of a root system
(see \cite{BOURB}), and $\mathcal{P}$ is subdivided in positive and negative
roots $\mathcal{P}^+, \mathcal{P}^-$ according to the chosen Weyl chamber,
where $\mathcal{P}^+ = \{ e_H, H \in \mathcal{A} \}$, then the representation described here is isomorphic to one on $\C \mathcal{P}^+$
described by $w.f_{H} = f_{w(H)}$ if $w.e_H \in \mathcal{P}^+$ and
$w.f_H = - f_{w(H)}$ if $w.e_H \in \mathcal{P}^-$, where $f_H$ denotes the basis element
of $\C \mathcal{P}^+$ corresponding to  $e_H \in \mathcal{P}^+$.

Finally, we notice that, when $W$ is a Coxeter group, then the representation $R_h$
for arbitrary $h$ factorizes through the extended Coxeter group $B/(P,P)$ introduced
by J. Tits in \cite{TITS}.

\medskip

We give in the following table the decomposition in irreducibles
of $R_0,R_1$ for the classical Coxeter groups of type $A_n, B_n, D_n$.
We label as usual irreducible representations of $\mathfrak{S}_n$
by partitions of size $n$ (with the convention that $[n]$ is the trivial
representation), of $W$ of type $B_n$ by couples of partitions $(\la,\mu)$
of total size $n$, and denote $\{ \la ,\mu \}$ the restriction of
$(\la,\mu)$ to the usual index-2 subgroup of $W$ of type $D_n$. Recall
that $\{ \la, \mu \} = \{ \mu , \la \}$ is irreducible
if and only if $\la \neq \mu$.

$$
\begin{array}{|l||l|}
\hline
 & R_0  \\
\hline
\hline
A_n, n \geq 3& [n-1,2]+[n,1] + [n+1]   \\
B_n, n \geq 4 & ([n-2,2],\emptyset) + ([n-2],[2]) + 2([n-1,1],\emptyset) + 2([n],\emptyset)  \\
B_3  & ([1],[2]) + 2 ([2,1],\emptyset) +  2([3],\emptyset)  \\
D_n, n \geq 4 & \{ [n-2,2],\emptyset\} + \{ [n-2],[2] \} + \{ [n-1,1],\emptyset \} + \{ [n],\emptyset \} \\
\hline
\hline
 & R_1 \\
\hline
\hline
A_n, n \geq 3& [n-1,1,1]+[n,1]    \\
B_n, n \geq 3 & ([n-2,1],[1]) + 2([n-1],[1])   \\
D_n, n \geq 4 & \{ [n-2,1],[1]\} + \{ [n-1],[1] \}  \\
\hline
\end{array}
$$
We sketch a justification of this table. For small values of $n$,
we prove this by using the character table. Then we use induction
with respect to a natural parabolic subgroup $W_0$ in the same
series, for which the branching rule is well-known. Restrictions
of $R_0$ and $R_1$ to this parabolic subgroup are then
isomorphic to the sum of the corresponding
representation $R_0$ or $R_1$ of the subgroup, plus
the permutation action of the reflections in $W$
which do not belong to $W_0$ (this is clear
for $R_0$, and a consequence of proposition \ref{proprestrparab} for $R_1$).
The decomposition in irreducibles of this permutation
representation is easy, namely $[n-1,1]+[n]$ for $A_n$,
$([n-2],[1]) + ([n-2,1],\emptyset) + 2([n-1],\emptyset)$
for $B_n$ and $\{ [n-2],[1] \} + \{[n-2,1],\emptyset\} + \{ [n-1],\emptyset \}$
for $D_n$. This provides the
restrictions of $R_0$ and $R_1$ to $W_0$.
From the combinatorial branching rule it
is easy to check that, for say $n \geq 5$,
only the given decompositions admit these
restrictions.

\section{Tables for $\kappa(W)$}

We compute here the value of $\kappa(W)$ for all irreducible
reflection groups $W$. More precisely, we compute all $d \in \Z$
such that there exists $w \in W$ and $H \in \mathcal{A}$ with
$w.e_H = \zeta e_H$ and $\zeta$ of order $d$. We call these
integers the $\mathcal{A}$-indices of $W$

Recall that the group $G(de,e,r)$ for $r \geq 2$ is defined as the
set of $r \times r$ monomial matrices with nonzero entries in $\mu_{de}(\C)$,
such that the product of these nonzero entries lie in $\mu_d(\C)$.

\begin{prop} The $\mathcal{A}$-indices of $W = G(de,e,r)$ are exactly
the divisors of $\kappa(W)$. Moreover, $\kappa(W) = de$
if $d \neq 1$ or $r \geq 3$. If $W = G(e,e,2)$
then $\kappa(W) = 2$.
\end{prop}
\begin{proof}
Since $G(e,e,2)$ is a Coxeter (dihedral) group, we can assume
$d \neq 1$ or $r \geq 3$. First note that the standard
hermitian scalar product on $\C^r$ is invariant under $W$.
We introduce the hyperplane arrangement
$$
\mathcal{A}^0_{de,r} = \{ z_i - \zeta z_j = 0\ | \ \zeta \in \mu_{de}(\C)
$$
We have $\mathcal{A}^0_{de,r} \subset \mathcal{A}$, and the orthogonal
to $H : z_i - \zeta z_j = 0$ is spanned by $e_H = e_i -\zeta^{-1}e_j$,
if $e_1,\dots,e_n$ denotes the canonical basis of $\C^r$. Let
$w \in W$. Since $w$ is a monomial matrix, there exists $\la_1,\dots,\la_r \in
\mu_{de}(\C)$ with $\la_i \in \mu_{de}(\C), \prod \la_i \in \mu_d(\C)$,
and $\sigma \in \mathfrak{S}_r$
such that $w.e_i = \la_i e_{\sigma(i)}$. 
Then $w. e_H = \mu e_H$ iff $\la_i e_{\sigma(i)} - \la_j \zeta^{-1} e_{\sigma(j)} =
\mu \la_i e_i + \mu \la_j e_j$. The two possibilities are
$\mu=1, \zeta = 1$ or $\mu \la_j = \la_i, \mu \la_i = \la_j \zeta^{-1}$,
that is $\mu^2 = \zeta^{-1}, \mu = \la_i \la_j^{-1}$.
It follows that $\mu \in \mu_{de}(\C)$. Conversely, assume we choose $\mu \in
\mu_{de}(\C)$, and let $\zeta = \mu^{-2}$. If $r \geq 3$ we can define
$w \in W$ by $\sigma = (1\ 2)$, $\la_2=1$, $\la_1 = \mu$, $\la_3 = \mu^{-1}$,
$\la_k = 1$ for $k \geq 4$, and $w.e_H = \mu e_H$ for $H : z_1 - \zeta z_2 = 0$.
We have $\mathcal{A} = \mathcal{A}^0_{de,r}$ when $d = 1$, so this
settles this case and we can assume $d \neq 1$. In that
case, $\mathcal{A} = \mathcal{A}^0_{de,r} \cup \mathcal{A}^+_r$,
where $\mathcal{A}^+_r$ is made out the hyperplanes $H_i : z_i = 0$,
whose orthogonals are spanned by the $e_i$. If $w.e_i=\mu e_i$
for $w \in W$ we obviously have $\mu \in \mu_{de}(\C)$, and conversely
if $\mu \in \mu_{de}(\C)$ we can define $w \in W$ by $w.e_1 = \mu e_1,
w.e_2 = \mu^{-1} e_2$ and $w.e_i = e_i$ for $i \geq 3$. It follows
that in this case too the set of $\mathcal{A}$-indices is
the set of divisors od $de$.
\end{proof}
By noticing that $G(2,1,r)$, $G(2,2,r)$ and $G(e,e,2)$,
are Coxeter groups, this gives the following.

\begin{cor} For $W = G(de,e,r)$, we have $\kappa(W) = 2$ iff $W$ is Coxeter group, if and only if
$de = 2$ or $(d,r) = (1,2)$.
\end{cor}

By checking out the 34 exceptional reflection groups, we prove case
by case the following.

\begin{prop} Let $W$ be an irreducible complex
reflection group. The set of $\mathcal{A}$-indices
is exactly the set of divisors of $\kappa(W)$.
\end{prop}

The following table gives the value of $\kappa(W)$, where $W$ an
complex reflection group labelled by its Shephard-Todd number (ST).
$$
\begin{array}{||c|c||c|c||c|c||c|c||c|c||c|c||}
\mathrm{ST} & \kappa & \mathrm{ST} & \kappa & \mathrm{ST} & \kappa & \mathrm{ST} & \kappa & \mathrm{ST} & \kappa & \mathrm{ST} & \kappa \\
\hline
4 & 6 & 10&12 & 16 &10 & 22 & 4 & 28 & 2 & 34 & 6 \\
5  & 6& 11 &24 & 17 &20 & 23 & 2 & 29 &4 & 35 & 2 \\
6 & 12& 12 &2 & 18 &30 & 24 & 2 & 30 & 2 & 36 & 2 \\
7 & 12& 13 &8 & 19 &60 & 25 & 6& 31 & 4 & 37 & 2 \\
8 & 4& 14 &6 & 20 & 6 & 26 & 6& 32 & 6 &    &  \\
9 & 8& 15 &24 & 21 & 12 & 27 & 6& 33 & 6 &    & \\
\end{array}
$$

We remark that the only non-Coxeter irreducible reflection groups with $\kappa(W) = 2$
are $G_{12}$ and $G_{24}$. Like in the case of $G_{12}$, it is straightforward to check that it
is possible to choose the 21 linear forms $\alpha_H$ such that the linear map $\Phi :
\C \mathcal{A} \to S^2 V^*$ is a morphism of $W$-modules. This phenomenon is reminiscent
of the special properties of their ``root systems'' in the sense of \cite{COHEN}. We refer to \cite{SHI} \S 2 and
\S 4 for a detailed study of these special root systems of type $G_{12}$ and $G_{24}$. In particular,
convenient linear forms for $G_{24}$ are described in \cite{SHI}, \S 4.1.

As a consequence
of this case-by-case investigation, propositions \ref{propequivPhiCox}
and \ref{propPhiEqKap2} can
be enhanced in the following

\begin{theo} Let $W$ be an irreducible reflection group. The linear
forms $\alpha_H$ can be chosen such that $\Phi$
is a morphism of $W$-modules if and only if $\kappa(W) = 2$.
This is the case exactly when $W$ is a Coxeter group
or an exceptional reflection group of type $G_{12}$ or $G_{24}$.
\end{theo}

\end{document}